\documentclass[10pt]{amsart}
\usepackage{amsmath, amssymb, amsthm,  epsfig}
\theoremstyle{plain}
\newtheorem{thrm}{Theorem}[section]
\newtheorem*{thrm*}{Theorem}
\newtheorem{lemma}[thrm]{Lemma}
\newtheorem{prop}[thrm]{Proposition}
\newtheorem{cor}[thrm]{Corollary}
\theoremstyle{definition}
\newtheorem{dfn}[thrm]{Definition}
\theoremstyle{remark}
\newtheorem{rmrk}[thrm]{Remark}
\theoremstyle{example}

\numberwithin{equation}{section}
\usepackage{color}

\begin{document}

\newcommand{\tx}{\tilde x}
\newcommand{\R}{\mathbb R}
\newcommand{\N}{\mathbb N}
\newcommand{\C}{\mathbb C}
\newcommand{\lie}{\mathcal G}
\newcommand{\hN}{\mathcal N}
\newcommand{\D}{\mathcal D}
\newcommand{\A}{\mathcal A}
\newcommand{\B}{\mathcal B}
\newcommand{\sL}{\mathcal L}
\newcommand{\sLi}{\mathcal L_{\infty}}

\newcommand{\G}{\Gamma}
\newcommand{\x}{\xi}

\newcommand{\eps}{\epsilon}
\newcommand{\al}{\alpha}
\newcommand{\be}{\beta}
\newcommand{\p}{\partial}  
\newcommand{\lig}{\mathfrak}

\def\dist{\mathop{\varrho}\nolimits}

\newcommand{\BCH}{\operatorname{BCH}\nolimits}
\newcommand{\Lip}{\operatorname{Lip}\nolimits}
\newcommand{\Hol}{C}                             
\newcommand{\lip}{\operatorname{lip}\nolimits}
\newcommand{\capQ}{\operatorname{Cap}\nolimits_Q}
\newcommand{\pCap}{\operatorname{Cap}\nolimits_p}
\newcommand{\Om}{\Omega}
\newcommand{\om}{\omega}
\newcommand{\half}{\frac{1}{2}}
\newcommand{\e}{\epsilon}
\newcommand{\vn}{\vec{n}}
\newcommand{\X}{\Xi}
\newcommand{\tLip}{\tilde  Lip}

\newcommand{\Span}{\operatorname{span}}

\newcommand{\ad}{\operatorname{ad}}
\newcommand{\Hm}{\mathbb H^m}
\newcommand{\Hn}{\mathbb H^n}
\newcommand{\Hone}{\mathbb H^1}
\newcommand{\Lie}{\mathfrak}
\newcommand{\Layer}{V}
\newcommand{\hgrad}{\nabla_{\!H}}
\newcommand{\im}{\textbf{i}}
\newcommand{\nz}{\nabla_0}
\newcommand{\s}{\sigma}
\newcommand{\se}{\sigma_\e}

\newcommand{\ued}{u^{\e,\delta}}
\newcommand{\ueds}{u^{\e,\delta,\sigma}}
\newcommand{\tnabla}{\tilde{\nabla}}

\newcommand{\bx}{\bar x}
\newcommand{\by}{\bar y}
\newcommand{\bt}{\bar t}
\newcommand{\bs}{\bar s}
\newcommand{\bz}{\bar z}
\newcommand{\btau}{\bar \tau}

\newcommand{\LC}{\mbox{\boldmath $\nabla$}}
\newcommand{\Ne}{\mbox{\boldmath $n^\e$}}
\newcommand{\nuo}{\mbox{\boldmath $n^0$}}
\newcommand{\nuu}{\mbox{\boldmath $n^1$}}
\newcommand{\nue}{\mbox{\boldmath $n^\e$}}
\newcommand{\nuek}{\mbox{\boldmath $n^{\e_k}$}}
\newcommand{\dse}{\nabla^{H\Su, \e}}
\newcommand{\dso}{\nabla^{H\Su, 0}}
\newcommand{\tX}{\tilde X}

\newcommand\red{\textcolor{red}}
\newcommand\green{\textcolor{green}}

\newcommand{\Xie}{X^\epsilon_i}
\newcommand{\Xje}{X^\epsilon_j}
\newcommand{\Su}{\mathcal S}
\newcommand{\F}{\mathcal F}

\title[Uniform Gaussian bounds and total variation flow ]{Uniform Gaussian bounds for subelliptic heat kernels and an application to the total variation flow of graphs over Carnot groups}

\author[L. Capogna]{Luca Capogna}
\address{Luca Capogna\\Department of Mathematical Sciences, University of Arkansas, Fayetteville, AR 72701, and Institute for Mathematics and its Applications, University of Minnesota, Minneapolis, MN 55455\\
}

\author[G.  Citti]{Giovanna  Citti}\address{Dipartimento di Matematica, Piazza Porta S. Donato 5,
40126 Bologna, Italy}\email{giovanna.citti@unibo.it}
\author[M.  Manfredini]{Maria  Manfredini}\address{Dipartimento di Matematica, Piazza Porta S. Donato 5,
40126 Bologna, Italy}\email{maria.manfredini@unibo.it}
\keywords{mean curvature flow, sub-Riemannian geometry, Carnot groups\\
The authors are partially funded by NSF awards  DMS 1101478 and DMS 0800522  (LC) and by CG-DICE ISERLES European Project (GC) and  by the University of Bologna:
funds for selected research topics (MM)}

\begin{abstract} In this paper we study  heat kernels associated to a Carnot group $G$, endowed with a family of collapsing   left-invariant Riemannian metrics $\sigma_\e$ which converge in the Gromov-Hausdorff sense to a sub-Riemannian structure
on $G$ as $\e\to 0$.  The main new  contribution are  Gaussian-type bounds on the heat kernel for the $\sigma_\e$ metrics which are stable as $\e\to 0$ and extend the previous time-independent estimates in \cite{CiMa-F}. As an application we study 
well posedness of  the  total variation flow of graph surfaces over a  bounded domain in  $(G,\s_\e)$. We establish  interior and boundary gradient estimates, and develop a Schauder theory which are stable as $\e\to 0$. As a consequence we obtain long time existence  of smooth solutions of  the  sub-Riemannian flow ($\e=0$), which in turn yield  sub-Riemannian minimal surfaces as $t\to \infty$.  \end{abstract}
\maketitle

\section{Introduction}
Models of image processing based on total variation flow $$\p_t u=div(\nabla u/|\nabla u|)$$ have been first introduced in by
Rudin, Osher, and Fatemi in \cite{ROF}, in order to perform edge preserving denoising. We refer the reader to the review paper \cite{CEPY} where  many recent applications of total variation equation in the Euclidean setting are presented.
The flow $t\to u(\cdot, t)$ represents the gradient descent  associated to the total variation energy $\int |\nabla u|$ and as such has the property that both the total variation of the solution $t\to u(\cdot, t)$ 
and the perimeter measure of fixed level sets $\{u(\cdot, t)=const\}$ are  non-increasing in time.
 Aside from its usefulness in image processing, the flow also arises in connection with the 
 limit of solutions of the parabolic $p-$Laplacian $\p_t u_p= div(|\nabla u_p|^{p-2} \nabla u_p)$
 as the parameter $p\to 1^+$. In the case where the evolution of graphs $S_t=\{(x,u(x,t))\}$ is considered, i.e. 
 $\p_t u = div (\nabla u /\sqrt{1+|\nabla u|^2})$, then in both the total variation flow and  the closely related {\it mean curvature flow} $\p_t u=\sqrt{1+|\nabla u|^2}div(\nabla u/\sqrt{1+|\nabla u|^2})$, given appropriate boundary/initial conditions, global in time solutions asymptotically converge to minimal graphs. For further results, references  and applications  of the Euclidean total variation flow we refer the reader to
 \cite{ABCM}, \cite{ACM}, \cite{BCN} and for an anisotropic version of the flow to \cite{moll}.
 
 \bigskip
 In this paper we  study  long time existence of graph solutions of the total variation flow 
 in a special class of degenerate Riemannian ambient spaces: The so-called (sub-Riemannian) 
 Carnot groups \cite{fol:1975}, \cite{ste:harmonic}. Such spaces are nilpotent Lie groups endowed with a metric structure  $(G,\sigma_0)$ that arises as limit of collapsing left-invariant {\it tame} Riemannian structures $(G,\sigma_\e)$. Our motivations are twofold:
 
 (a) On the one hand, minimal surfaces and mean curvature flow in the Carnot group setting (and in particular in the special case of the Heisenberg group) have been the subject of numerous papers in recent years, leading to partial solutions of long-standing problems such as the Pansu conjecture\footnote{See for instance  \cite{CDPT}, \cite{RR}, \cite{MR},  \cite{Pauls}, \cite{CHMY} \cite{DGNP} ,\cite{CHY}, \cite{CCM1}, \cite{CCM2}, \cite{serracassano} and references therein}. Despite such advances little is known about explicit construction of minimal surfaces. Since the asymptotic limit of graphical mean curvature flow provide minimal surfaces in the Riemannian setting it is only natural to follow the same approach in the sub-Riemannian setting. However this avenue runs into considerable (so far unsolved) technical difficulties due to the combined effect of the degeneracy of the metric and the non-divergence form aspect of the relevant PDE. In the present work we provide the basis for the construction of sub-Riemannian minimal graphs through the asymptotic behavior as $t\to \infty$
 of solutions of the sub-Riemannian total variation flow. For previous work on a similar theme
 see the work of Pauls \cite{Pauls} and Cheng, Hwang and Yang \cite{CHY}.

 (b)  A new class of image processing models based on the functionality of the visual cortex  have been recently introduced in Lie groups. With a generalization of the classical Bargmann transform studied by Folland  (see \cite{Folland}) which lifts a $L^2$ function to a new one, defined on the phase space, an image can be lifted to a Lie group with a sub-Riemannian metric. The choice of the Lie group depends on the geometric property of the image to be studied:   in \cite{CS},  and \cite{D1}, \cite{D2}  2D images are lifted to surfaces in the Lie group of Euclidean motions of the plane to study geometric properties of their level lines, in \cite{D3}  cardiac images  are lifted in the Heisenberg group to study their deformations, in \cite{BCCS} moving images are lifted  in the Galilei group. The image processing is performed in these groups with algorithms expressed in terms of second order subelliptic differential equations. In particular sub-Riemannian mean curvature flows or total variation flows can be applied to perform image completion or impainting in this setting.  
 
 \bigskip
 Our approach to the existence of global (in time) smooth solutions is based on a Riemannian approximation scheme:
We  study  graph solutions of the total variation flow
in the Riemannian spaces $(G,\sigma_\e)$ where $G$ is a Carnot group
and $\sigma_\e$ is a family of tame Riemannian metrics that 'collapse'  as $\e\to 0$ to a sub-Riemannian metric $\s_0$ in $G$. The {\it main technical novelties  in the paper} are a series of  different a-priori estimates, which are stable as $\e\to 0$:

\begin{enumerate} 
\item Heat kernel estimates for sub-Laplacians and their elliptic regularizations (see Proposition \ref{uniform heat kernel estimates}). These estimates are established in the setting of any Carnot group and provide a {\it parabolic} counterpart to the  time-independent estimates proved by two of the authors  in \cite{CiMa-F}.

\item Uniform Schauder estimates for second order, non-divergence form subelliptic PDE and their {\it elliptic} regularizations  (see Proposition \ref{ellek}). These estimates are established in the setting of any Carnot group.

\item Interior gradient estimates for solutions of the total variation flow  (see 
 Proposition \ref{u_t}). These estimates are established  in the setting of any Carnot group.

\item Boundary gradient estimates for solutions of the total variation flow  (see Proposition \ref{boundary estimates}). These estimates are established only in the setting of  Carnot group of step two.
\end{enumerate}
We remark that (3) and  (4) can also be proved, with similar arguments, for solutions of the sub-elliptic mean curvature flow\footnote{For more references on the sub-Riemannian mean curvature flow see \cite{CC}, \cite{Dragoni}, \cite{ccs} and \cite{Manfredi}}. The limitation to Carnot groups of step two in (4) is due to the fact that 
for step three and higher there no known suitable barrier functions, or in other words, the class of known explicit minimal graphs is not sufficiently large.

In order to state our results we need to introduce some notation.

\subsection{Carnot group structure}
Let $G$ be an analytic and simply connected Lie group with
topological dimension $n$ and such that its Lie algebra $\lie$
admits a stratification $\lie=V^1\oplus V^2 \oplus ...\oplus V^r$,
where $[V^1,V^j]=V^{j+1}$, if $j=1,\ldots,r-1$, and $[V^k,V^r]=0$,
$k=1,\ldots,r$. Such groups are called in \cite{fol:1975},
\cite{fs:hardy}, and \cite{ste:harmonic} {\sl stratified nilpotent
Lie groups}. Set $H(0)= V^1$, and for any $x\in G$ we let
$H(x)=xH(0)=\text{span}(X_{1},\ldots,X_{m})(x)$. The distribution
$x\to H(x)$ is called {\it the horizontal sub-bundle} $H$. 

The metric structure is given by assuming that one has a  
left invariant positive definite form $\s_0$ defined in the sub-bundle $H$. We fix a orthonormal horizontal frame $X_{1},\ldots,X_{m}$ which we complete it to a basis $(X_1,\ldots,X_{n})$ of $\lie$
by choosing for every $k=2, \ldots,  r$ a basis of $V_k$.  If $X_i$
belongs to $V_k$, then we will define its {\it homogeneous degree } as \begin{equation}\label{degree} d(i)=k .\end{equation}
We will denote by $xX
= \sum_{i=1}^{n} x_{i} X_{i}$ a generic element of $\lie$.
 Since the exponential map $\exp:\lie \to G$ is a
global diffeomorphism we use exponential coordinates in $G$, and
denote $x=(x_1, \ldots, x_n)$ the point $\exp \big(xX\big).$ We
also set $x_H = (x_1, \ldots, x_m)$.   Define non-isotropic
dilations as $\delta_s(x)= (s^{d(i)} x_{i})$, for $s>0$.

 We let
$\nabla_0=(X_1,\ldots,X_m)$ denote the {\it horizontal gradient}
operator. If $\phi\in C^{\infty}(G)$ we set $\nabla_0 \phi=\sum_{i=1}^m X_i\phi X_i$ and $|\nabla_0 \phi|^2=\sum_{i=1}^m (X_i \phi )^2$.

We denote by $(X_1,\ldots,X_{n})$ (resp. $(X^r_1,\ldots,
X^r_{n})$) the left invariant (resp. right invariant) translation of
the frames $(X_1,\ldots,X_{n})$.
 
   The vectors $X_{1},\ldots, X_{m}$ and their commutators span
all the Lie algebra
 $\lie$, and consequently verify H\"ormander's finite rank condition
(\cite{hormander}). This allows to use the results from
\cite{nsw},  and define a control distance $d_0(x,y)$
 associated to the distribution $X_{1},\ldots, X_{m}$, which is
called {\sl the
 Carnot-Carath\'eodory metric} (denote by $d_{r,0}$
 the corresponding right invariant distance). We
 call the couple $(G,d_0)$ a {\it Carnot Group}.

We define a family of left invariant Riemannian metrics $\s_{\e}$,
$\e>0$ in $\lie$ by requesting that  $$\{X_1^\e,\ldots,X_n^\e\}:=\{X_1,\ldots, X_m, \e
X_{m+1}
 ,\ldots, \e X_n\}$$ is an orthonormal
frame. We will denote by $d_{\e}$ the corresponding distance
functions. Correspondingly we use $\nabla_\e,$ (resp. $
\nabla^r_\e$) to denote the left (resp. right) invariant gradients. In particular,
if  $\phi\in C^{\infty}(G)$ we set $\nabla_\e \phi=\sum_{i=1}^m X_i^\e\phi X_i^\e$ and  $|\nabla_\e \phi|^2=\sum_{i=1}^n (X_i^\e \phi )^2$.

We conclude by recalling  the expression of the left invariant vector fields
in exponential coordinates  (see \cite{Roth:Stein})
\begin{equation}\label{bch}
X_i = \p_i + \sum_{k=d(i)+1}^r\sum_{d(j) = k}p_{ik}^j(x)\p_j,
\end{equation}
where $p_{ik}^j(x)$ is an homogeneous polynomial of degree
$k-d(i)$ and depends only on $x_h$, with $d(1)\leq d(h) \leq
k-d(i)$.

\subsection{The total variation}

The total variation flow is characterized by the fact that each point of the evolving surface graph moves in the direction of the  {\it upward} unit normal with speed equal to the mean curvature times the volume element.
In the setting of the approximating Riemannian metrics $(G,\s_\e)$ and in terms of the functions $t\to u(\cdot,t):\Om\subset G \to \R$ describing the evolving graphs, the relevant equation reads:

\begin{equation}\label{pdee}\frac{\p u_{\e}}{\p t} = h_{\e}=\sum_{i=1}^n X_i^{\e}\Big(\frac{X_i^{\e}u_{\e}}{W_{\e} }\Big)=\sum_{i,j=1}^na_{ij}^\e (\nabla_\e u_\e) X_i^\e X_j^\e u_\e
\end{equation}
for $x\in \Om$ and $t>0$, with $u_{\e}(x,0)=\varphi(x)$, $h_\e$ is the mean curvature of the graph of $u_\e (\cdot, t)$ and $$W_\e^2=1+|\nabla_\e u_\e|^2=
1+\sum_{i=1}^n (X_i^\e u_\e)^2 \text{ and } a_{ij}^\e(\xi)=W_\e^{-1}\Bigg(\delta_{ij}-\frac{\xi_i \xi_j}{1+|\xi|^2}\Bigg),$$ for all $\xi\in \R^n$.
In the sub-Riemannian limit $\e=0$ the equation reads
\begin{equation}\label{pde01}\frac{\p u}{\p t} =\sum_{i=1}^m X_i  \Big(\frac{X_iu}{\sqrt{1+|\nabla_0 u|^2}}\Big),
\end{equation} for $x\in \Om$ and $t>0$, with $u(x,0)=\varphi(x)$.

We will be  concerned with uniform (in the parameter $\e$ as $\e\to 0$)  estimates and with the asymptotic behavior of solutions
to the initial value problem for the mean curvature motion of graphs
over bounded domains of a Carnot group $G$,
 \begin{equation}\label{ivp}
 \Bigg\{
 \begin{array}{ll}
 \p_t u_\e= h_\e   &\text{ in }Q=\Om\times(0,T) \\
 u_\e=\varphi &\text{ on } \p_p Q.
 \end{array}
 \end{equation}
 Here $\p_p Q=(\Om\times \{t=0\})\cup (\p\Om \times (0,T))$ denotes the parabolic boundary of $Q$.

  Given appropriate hypothesis on the   data,  for instance convexity of $\Om$ and $\varphi$ twice continuously differentiable, the classical parabolic theory yields existence and uniqueness for smooth solutions $u_\e$ of \eqref{ivp}. However classical parabolic theory will only provide estimates involving constants that degenerate as $\e\to 0$ (in the transition from parabolic to degenerate parabolic regime). Our main goal consists in proving stable estimates.
  
  Our first result consists in showing that if the initial/boundary data is sufficiently smooth then the solutions of \eqref{ivp} are Lipschitz up to the boundary uniformly in $\e>0$.
  
\begin{thrm}{(Global gradient bounds)}\label{global estimates}
Let $G$ be a Carnot group of step two,
 $\Om\subset G$ a bounded, open, convex\footnote{We say that a set $\Om\subset G$ is convex in the Euclidean sense if
$\exp^{-1}(\Om)\subset \lie$ is convex in the Euclidean space $\lie$.
In a group of step two this condition is translation invariant.} set and $\varphi\in C^2(\bar{\Om})$. For $1\ge\e>0$ denote by  $u_\e\in C^2(\Om\times(0,T))\cap C^1(\bar\Om\times(0,T))$ the non-negative unique  solution of the initial value problem \eqref{ivp}. There exists $C=C(G, ||\varphi||_{C^2(\bar\Om)})>0$ such that
 \begin{equation}\label{b-1.1a}
 \sup_{\bar\Om\times(0,T)}|\nabla_\e u_\e| \le  \sup_{\bar\Om\times(0,T)}|\nabla_1 u_\e| \le C.
 \end{equation}
 \end{thrm}

\begin{rmrk}\label{rmrk-b}
The hypothesis of the Theorem above can be slightly weakened by asking that
$G$ is a step two Carnot group, $\varphi\in C^2(\bar\Om)$ but
that $\p\Om$ be required to be  convex  in the Euclidean sense only  in a neighborhood of its characteristic locus  $\Sigma(\p\Om)$ and that its mean curvature $h_\e^{\p\Om} \le -\delta<0$ in $\p\Om\setminus \Sigma(\p\Om)$.
\end{rmrk}

Having established Lipschitz bounds, it is easy to show that the right  derivatives $X^r_i u_\e$ of the solutions of \eqref{ivp} are themselves solutions of \eqref{diff-eq}, a divergence form, degenerate parabolic PDE whose weak solutions satisfy a Harnack inequality (see \cite{CCR} and Proposition \ref{propCCR} below). Consequently one obtains $C^{1,\alpha}$ interior estimates for the solution $u_\e$ which are uniform in $\e>0$. At this point one rewrites the PDE in \eqref{ivp} in  non-divergence form\footnote{this is possible since for $\e>0$ the solutions $u_\e$ are sufficiently smooth} of the total variation flow equation $\p_tu_\e=a_{ij}^\e (\nabla_\e u_\e) X_i^\e X_j^\e u_\e$ and invokes the 
Schauder estimates in Proposition \ref{ellek} to prove local higher regularity and long time existence. Since all the estimates are stable as $\e\to 0$ one immediately obtains smoothness and the consequent global in time existence of the solution for the sub-Riemannian case $\e=0$.

\begin{thrm} \label{global in time  existence results}
In the hypothesis of Theorem \ref{global estimates} (or Remark \ref{rmrk-b}) one has that there exists a unique solution
$u_\e \in C^{\infty}(\Om\times (0,\infty))\cap L^\infty((0,\infty),C^1(\bar \Om))$ of the initial value problem
\begin{equation}\label{ivp-inf}
 \Bigg\{
 \begin{array}{ll}
 \p_t u_\e= h_\e    &\text{ in }Q=\Om\times(0,\infty) \\
 u_\e=\varphi &\text{ on } \p_p Q 
 \end{array}
 \end{equation}
and that for each $k\in \N$ there exists $C_k=C_k(G,\varphi,k,\Om)>0$ not depending on $\e$ such that
\begin{equation}\label{stable estimates}
||u_\e||_{C^k(Q)} \le C_k.
\end{equation}
\end{thrm}

\bigskip

Since the estimates are uniform in $\epsilon$ and in time,
and with respect to $\epsilon$, we will deduce the following corollary:

\begin{cor} 
Under the assumptions of the Theorem \ref{global estimates},
as $\e\to 0$ 
the solutions $u_\e$ converge uniformly (with all its derivatives) on compact subsets of $Q$ to the unique,  smooth solution 
$u_0\in C^{\infty}(\Om\times (0,\infty))\cap L^\infty((0,\infty),C^1(\bar \Om))$ of the sub-Riemannian total variation flow  \eqref{pde01} in $\Om\times (0,\infty)$ with initial data $\varphi$.
\end{cor}

\begin{cor} 
Under the assumptions of Theorem \ref{global estimates},
as $T\to \infty$
the solutions $u_\e(\cdot, t)$ converge uniformly on compact subsets of $\Om$ to the unique  solution
of the minimal surface equation $$h_\e=0\quad in \; \Om$$ with boundary value $\varphi$,
while $u_0=\lim_{\e\to 0} u_\e\in C^\infty (\Om)\cap Lip(\bar \Om)$ is  the unique 
solution of the
{\it sub-Riemannian minimal surfaces} equation
$h_0=0$  in $\Om$, with boundary data $\varphi$.
\end{cor}
%
%


\section{Structure stability in the Riemannian limit}

The Carnot-Carath\'eodory metric  is equivalent
to a more explicitly defined pseudo-distance function, that we
will call (improperly) {\sl gauge distance}, defined as
\begin{equation}\label{gauge}
|x|^{2r!}=\sum_{k=1}^r \sum_{i=1}^{m_k} |x_{i }|^{\frac{2r!}{d(i)}}, \text{ and
} d_0(x,y)=|y^{-1}x|.
\end{equation}
The {\it ball-box} theorem in \cite{nsw} states that there exists $A=A(G,\sigma_0)$ such that for each $x\in G$,
$$A^{-1}|x|\le d_0(x,0)\le A|x|.$$
If $x\in G$ and $r>0$, we will denote by
 $$B(x,r)=\{y\in G \ |\ d_0(x,y)<r\}$$ the   balls
in the gauge distance. For each $\e>0$ we also define the distance function $d_\e$ corresponding to the Riemannian metric $\sigma_\e$,
$$d_\e(x,y)=\inf\{ \int_0^1 |\gamma'|_{\sigma_\e}(s) ds \text{ with }\gamma:[0,1]\to G \quad\quad\quad\quad\quad\quad\quad\quad $$$$\quad\quad\quad\quad\quad\quad\quad\quad\quad\quad\quad\quad
\text { a Lipschitz curve s. t. } \gamma(0)=x, \gamma(1)=y\}$$
as well as the  pseudo-distance
$d_{G,\e} (x,y)=N_\e (y^{-1}x)$ with
\begin{equation}
N_\e^2(x)=\sum_{d(i)=1} x_i^2+
 \min\bigg\{ \sum_{i=2}^r
 (
 \sum_{d(k)=i}
  x_k^2
 )^{\frac{1}{i}},
  \e^{-2}\sum_{d(i)\ge 2} x_i^2
 \bigg\}.
\end{equation}
Note that  in the definition of $d_\e$, if the curve for which the infimum is achieved happens to be horizontal
then $d_e(x,y)=d_0(x,y)$. In general we have $\sup_{\e>0} d_\e(x,y)= d_0(x,y)$ and
it is well known\footnote{See for instance \cite{gro:metric} and references therein} that $(G,d_{\e})$  converges
in the Gromov-Hausdorff sense as $\e\to 0$ to the sub-Riemannian
space $(G,d_0)$. 
The {\it ball-box} theorem in \cite{nsw} and elementary considerations yield that there exists $A=A(G,\sigma_0)>0$ independent of $\e$ such that for all $x,y\in G$
\begin{equation}\label{ball-box}
A^{-1} d_{G,\e}(x,y) \le d_\e (x,y) \le A d_{G,\e}(x,y) 
\end{equation}
(see for example \cite{CCR}).
\subsection{Stability of the homogenous structure  as $\e\to 0$}

 If $G$ is a Carnot group, $d_\e$ is the distance function associated to $\sigma_\e$, we will denote
$$B_\e(x,r)=\{y\in G| d_\e(x,y)<r\}.$$
If we denote by $dx$ the Lebesgue measure and by $|\Om|$ the corresponding measure of a subset $\Om$, then Rea and two of the authors have recently proved in \cite{CCR}
that \begin{prop}\label{homog-stab}
There is a constant $C $ independent of $\epsilon$ such that for every $x\in G$ and $r>0$,
$$|B_\e(x,2r)|\le C |B_\e(x,r)|.$$
 \end{prop}
Having this property the spaces $(G,d_\e,dx)$ are called {\it  homogenous} with constant $C>0$ independent of $\e$ (see \cite{cw:1971}).

Let $\tau>0$ and consider the space $\tilde G = G\times (0,\tau)$ with its  product Lebesgue measure $dxdt$. In $\tilde G$ define the pseudo-distance function
\begin{equation} \label{defde}\tilde d_{\e}( (x,t), (y,s))= \max( d_\e(x,y), \sqrt{|t-s|} ).
\end{equation}
 Proposition \ref{homog-stab} tells us that
$(\tilde G, \tilde d_{\e},dxdt)$ is a homogeneous space with constant independent of $\e\ge 0$. 

In the paper \cite{CCR} it is also shown that a Poincar\'e inequality holds with a choice of a constant which is stable as $\e\to 0$.
The stability of the homogenous space structure and of the constant in the Poincar\'e inequality  are two of the key factors in the proof of Proposition \ref{propCCR}.
\subsection{Stability in the estimates of the Heat Kernel} \label{GaussGauss}The results in this section are of independent interest and concern uniform Gaussian estimates for the heat kernel associated to elliptic regularizations of the Carnot group sub-Laplacians. They will be used in this paper in connection to the Schauder estimates and the higher regularity of the total variation flow.
We will deal with non-divergence form operators similar to those in  (\ref{pdee}),  but with 
constant coefficients. More precisely we will consider the operator 
  \begin{equation}\label{operatore}L_{\epsilon, A} = \p_t - \sum_{i,j=1}^n a_{ij} X_i^\e X_j^\e  \end{equation}
where $A=( a_{ij})_{ij=1, \ldots n}$  is a symmetric,  real-valued $n\times n$ matrix, such that  for  some choice of constants $\Lambda, C_{1}, C_{2}>0$ and for all $\xi\in \R^n$ one has
\begin{equation}\label{emmeLambda}\Lambda^{-1} \sum_{d(i)=1} \xi_i^2 + C_{1} \sum_{d(i)>1} \xi_i^2  \leq \sum_{i,j=1}^n a_{ij} \xi_i \xi_j \leq \Lambda \sum_{d(i)=1} \xi_i^2 + C_{2} \sum_{d(i)>1} \xi_i^2. \end{equation}\label{unif}
The point here is that the constants $C_1,C_2, \Lambda$ are independent of $\e$ and provide {\it $\Lambda-$uniform coercivity  of \eqref{operatore} in the horizontal directions}. Formally, in the sub-Riemannian limit $\e\to 0$  the equation becomes 
 \begin{equation}\label{operator}L_A = \p_t - \sum_{i,j=1}^ma_{ij} X_i X_j.\end{equation}
In order to ensure that  the operator $L_{\e,A}$ (reps. $L_A$) is uniformly elliptic (reps. subelliptic), we will assume that the 
matrix $A=(a_{ij})$ belongs to a  set of the form 
$$ M_{\Lambda} = \{A:  A \; \text{is a } \;  \text{ symmetric } n\times n, \text{ real valued constant matrix,} \quad \quad $$
$$\quad \quad \quad \quad \text{ satisfying }\eqref{emmeLambda} \text{ for some choice of } C_{1} \text{ and }C_{2}\}$$
for  some fixed $\Lambda>0$.

Heat kernel estimates in Nilpotent groups are well known (see for instance \cite{MR865430} and references therein). 
We also refer to \cite{BLU1} where a self contained proof is provided. In our work we will need estimates which are uniform in the variable $\e$ as $\e\to 0$, in the same spirit as the results in \cite{CiMa-F}. 
We will denote $\Gamma_{\e,  A}(x,t)$ the fundamental solution of (\ref{operatore}), with matrix $(a_{ij})$
in $M_{\Lambda}$, and  $\Gamma_{  A}(x,t)$  the fundamental solution of (\ref{operator}).

\begin{prop}\label{uniform heat kernel estimates}
There exists constants $C_\Lambda>0$  depending on $G,\sigma_0, \Lambda$ but independent of $\e$ such that for each $\e>0$, $x\in G$ and $t>0$ one has
\begin{equation}
C_\Lambda^{-1}\frac{e^{-C_\Lambda\frac{d_\e(x,0)^2}{t}}} {|B_\e (0, \sqrt{t})|}\le \Gamma_{\e, A}(x,t)\le C_\Lambda\frac{e^{-\frac{d_\e(x,0)^2}{C_\Lambda t}}} {|B_\e (0, \sqrt{t})|}.
\end{equation}
For $s\in \N$ and $k-$tuple $(i_1,\ldots,i_k)\in \{1,\ldots,n\}^k$ there exists $C_{s,k}>0$ depending only on $k,s,G,\sigma_0, \Lambda$ such that
\begin{equation}
|(\p_t^s X^\e_{i_1}\cdots X^\e_{i_k} \Gamma_{\e, A})(x,t)| \le C_{s,k} t^{-s-k/2} \frac{e^{-\frac{d_\e(x,0)^2}{C_\Lambda t}}} {|B_\e(0, \sqrt{t})|}
\end{equation}
for all $x\in G$ and $t>0$. For any $A_1,A_2\in M_\Lambda$, $s\in \N$ and $k-$tuple $(i_1,\ldots,i_k)\in \{1,\ldots,m\}^k$ there exists $C_{s,k}>0$ depending only on $k,s,G,\sigma_0, \Lambda$ such that
\begin{equation}
|(\p_t^s X_{i_1}\cdots X_{i_k} \Gamma_{\e, A_1})(x,t) - (\p_t^s X_{i_1}\cdots X_{i_k} \Gamma_{\e, A_2})(x,t) |\leq
\end{equation}
$$ \le ||A_1 - A_2||C_{s,k} t^{-s-k/2} \frac{e^{-\frac{d_\e(x,0)^2}{C_\Lambda t}}} {|B_\e(0, \sqrt{t})|},$$ where $||A||^2:=\sum_{i,j=1}^n a_{ij}^2$.

Moreover, for $s\in \N$ and $k-$tuple $(i_1,\ldots,i_k)\in \{1,\ldots, m\}^k$ as $\e\to 0$ one has 
\begin{equation}\label{GetobarG}{ X}_{i_1}\cdots { X}_{i_k} \p_t^s  \Gamma_{\e, A}\to {X}_{i_1}\cdots {X}_{i_k}\p_t^s \Gamma_{A}\end{equation}
uniformly on compact sets and  in a dominated way on all $G$.

\end{prop}

\bigskip

The proof of our result is directly inspired from the proof in \cite{CiMa-F},  where the time independent case was studied by two of us, and  where a general procedure was introduced for 
handling the dependence on $\epsilon$ and obtain independent estimates in the sub-Laplacian case.

  Let us consider the group $\bar G= G\times G$  defined in terms of $n$ new coordinates $y\in G$, so that points of $\bar G$ will be denoted  $\bar x =(x,y)$.  Denote
by $Y_1,\ldots,Y_n$ a copy of the vectors $X_1,\ldots,X_n$, defined in terms of new variables $y$. The vector fields $(X_i)$ and $(Y_i)$ are defined in the product algebra $\bar \lie = \lie\times \lie$. 

In $\bar G$ we will consider two different families of sub-Riemannian structures:
\begin{itemize} \item  A  sub-Riemannian structure determined by the choice of horizontal vector fields given by
\begin{equation}\label{frame1} ({\bar X}^0_1, \cdots, {\bar X}^0_{m+n}) = (  X_1,\ldots,X_m, Y_1,\ldots,Y_n).\end{equation}
The sub-Laplacian/heat operator associated to this structure is 
\begin{equation}\label{giovanna}
\bar L_{0,A}=\p_t -\sum_{i,j=1}^{m+n}\bar a_{ij} \bar X^0_{i}\bar X^0_{j}=\p_t -\sum_{i,j=1}^ma_{ij} X_{i}X_{j} -\sum_{i,j=1}^na_{ij}  Y_{i}Y_j. \end{equation}
We complete the horizontal frame  to a basis of the whole Lie algebra by adding the commutators of the vectors $(X_i)$: 
$(X_{m+1}, \ldots X_{n})$. 
We denote the  exponential coordinates  of a point $\bar x$ around a point $\bar x_0$ in terms of the full frame through the coefficients $(v^0_i, w^0_i)$ which are defined by
$$\bar x = exp(\sum_{i=1}^m v^0_i X_i + \sum_{i=1}^n w^0_i Y_i + \sum_{i=m+1}^n v^0_i X_i) (\bar x_0 ).$$

\item For every $\e\in [0,1)$ consider a 
sub-Riemannian structure determined by the choice of horizontal vector fields given by
\begin{equation}\label{frame2}
({\bar X}_1^\e, \cdots, {\bar X}_{m+n}^\e)= ( X_1,\ldots,X_m, Y_1,\ldots,Y_m, X_{m+1}^\e+Y_{m+1},\ldots, X_{n}^\e+Y_n).\end{equation}
The sub-Laplacian/heat operator associated to this structure is 
$$
\bar L_{\e, A}=
\p_t-\sum_{i,j=1}^ma_{ij} X_{i}X_{j} -\sum_{i,j=1}^ma_{ij} Y_{i}Y_{j} -\sum_{i>m \text{ or } j>m}a_{ij} (X^\e_i + Y_{i})(X^\e_ j + Y_{j}).$$
Analogously, the exponential coordinates associated to   $L_{\e, A}$ will depend on the given horizontal frame \eqref{frame2} and the family 
$(X_{m+1}, \ldots X_{n})$. The coordinates of a point $\bar x$  are the coefficients $v_i^\e$ and $w_i^\e$  satisfying
\begin{equation}\label{canonical}
\bar x = exp\Big(\sum_{i=1}^m v^\e_i X_i + \sum_{i=1}^m w^\e_i Y_i + \sum_{i=m+1}^n w^\e_i (X^\e_i +Y_i) +\sum_{i=m+1}^n v^\e_i X_i \Big) (\bar x_0 ).\end{equation}
 
\end{itemize}

We denote by  $\bar \Gamma_{0,A}$ and $\bar \Gamma_{\e, A}$ the heat kernels of the corresponding heat operators.
In both structures we define  associated  (pseudo)distance functions $\bar d_0 (\bar x, \bar y)$ and  $\bar d_\e(\bar x, \bar y)$ that are equivalent to those defined in \eqref{gauge} and that assign unit weight  to the corresponding horizontal vectors  while $(X_{m+1}, \ldots X_{n})$
will be weighted according to their degree $d(i)$. \begin{dfn}
For every $\e> 0$ and  $\bar x, \bar x_0 \in \bar G$  define
$$\bar d_\e(\bar x, \bar x_0)= \sum_{i=1}^{m}|v^\e_i|  + \sum_{i=1}^{m}|w^\e_i|  +  \sum_{i=m+1}^{n} \min(|w^\e_i|, |w^\e_i|^{1/d(i)} )+ \sum_{i=m+1}^{n}| v^\e_i|^{1/d(i )}$$
For $\e=0$ and  $\bar x, \bar x_0 \in \bar G$ define
$$\bar d_0(\bar x, \bar x_0)= \sum_{i=1}^{n} (|w^0_i|^{1/d(i)} + | v^0_i|^{1/d(i )})$$

We will denote by $\bar B_\e$ and $\bar B_0$ the corresponding metric balls.

\end{dfn}

\begin{lemma}\label{convergencelem1}
There exists constants $C_\Lambda>0$  depending on $G,\sigma_0, \Lambda$ but independent of $\e$ such that for each $\e>0$, $\bar x\in\bar G$ and $t>0$ one has
\begin{equation}\label{gamma}
C_\Lambda^{-1}\frac{e^{-C_\Lambda\frac{\bar d_\e(\bar x,0)^2}{t}}} {|\bar B_\e (0, \sqrt{t})|}\le \bar \Gamma_{\e, A}(\bar x,t)\le C_\Lambda\frac{e^{-\frac{\bar  d_\e(\bar x,0)^2}{C_\Lambda t}}} {|\bar  B_\e (0, \sqrt{t})|}.
\end{equation}
For $s\in \N$ and $k-$tuple $(i_1,\ldots,i_k)\in \{1,\ldots,n + m\}^k$ there exists $C_{s,k}>0$ depending only on $k,s,G,\sigma_0, \Lambda$ such that
\begin{equation}\label{Xgamma}
|(\p_t^s {\bar  X}^\e_{i_1}\cdots{\bar  X}^\e_{i_k} \bar \Gamma_{\e, A})(\bar x,t)| \le C_{s,k} t^{-s-k/2} \frac{e^{-\frac{\bar d_\e(\bar x,0)^2}{C_\Lambda t}}} {|\bar B_\e(0, \sqrt{t})|}
\end{equation}
for all $\bar x\in \bar G$ and $t>0$.
Moreover, for $s\in \N$ and $k-$tuple $(i_1,\ldots,i_k)\in \{1,\ldots, m\}^k$ as $\e\to 0$ one has 
\begin{equation}\label{barGetobarG}{ \bar X}^\e_{i_1}\cdots {\bar X}^\e_{i_k} \p_t^s \bar \Gamma_{\e, A}\to { \bar X}^0_{i_1}\cdots {\bar X}^0_{i_k}\p_t^s \bar\Gamma_{0,A}\end{equation}
uniformly on compact sets, in a dominated way on all $\bar G$.
\end{lemma}

\begin{proof}
In order to estimate the fundamental solution of the operators $\bar L_{\e,A}$ in terms of $\bar L_{0,A} $, we define a volume preserving change of variables on the Lie algebra and on $G$:
Define $\bar F_\e:\bar G\to \bar  G$ as  $\bar F_\e (\bar x)= \exp( T_\e( \log(\bar x)) $ with
$$T_\e ({\bar  X}^0_i)={\bar X}^\e_i\text{ for } i=1,\ldots,m+n.$$ 
 By definition the relation between the distances $\bar d_0$ and $\bar d_\e$ is expressed by the formula
\begin{equation}\label{distances}\bar d_\e (\bar x, \bar x_0)=\bar d_0( \bar F_\e (\bar x), \bar  F_\e(\bar x_0)).\end{equation}

Analogously we also have
\begin{equation}\label{changegamma}\bar \Gamma_{\e, A} (\bar x,t)=\bar \Gamma_{0,A}(\bar F_\e(\bar x), t),
\end{equation}
$$ X_i\bar \Gamma_{\e, A} (\bar x,t)= X_i \bar \Gamma_{0,A}(\bar F_\e(\bar x), t),\text{ for } i=1, \cdots,  m$$and$$
 (X_i^\e + Y_i)\bar \Gamma_{\e, A} (\bar x,t)= Y_i \bar \Gamma_{0,A}(\bar F_\e(\bar x), t),\text{ for } i=m+1, \cdots,  n ,$$
with similar identities holding for the  iterated derivatives.

In view of the latter, assertions (\ref{gamma}) and (\ref{Xgamma}) immediately follow from the well known estimates of $\bar\Gamma_{0,A}$ (see for instance there references cited above or  \cite{BLU1}).
The pointwise convergence (\ref{barGetobarG}) is also an immediate consequence of  \eqref{distances} and \eqref{changegamma}.
In order to prove the dominated convergence result we need to relate the distances $\bar d_0$ and $\bar d_\e$: Expressing the exponential coordinates $v_i^\e,w_i^\e$ in terms of $v_i^0, w_i^0$ one easily obtains
 $$\bar d_\e( \bar x,  \bar x_0)= \sum_{i=1}^m(|v_i^0| + |w_i^0|) + \sum_{i=m+1}^n \Big( |v_i^0- \e w_i^0|^{1/d(i)} + \min(|w_i^0|, |w_i^0|^{1/d(i)} ) \Big)$$
so that for all\footnote{This estimate indicates the well known fact that at large scale the Riemannian approximating distances are equivalent to the sub-Riemannian distance}
  $\bar x, \bar x_0\in \bar G$ \begin{equation}
\label{tocheck}
\bar d_0(\bar x, \bar x_0) - C_0\leq \bar d_\e( \bar x, \bar x_0)\leq \bar d_0(\bar x, \bar x_0)+ C_0\end{equation}
where $C_0$ is independent of $\e$. The latter and (\ref{Xgamma}) imply dominated convergence on $\bar G$.
\end{proof}

\begin{lemma}\label{convergencelem}
For $s\in \N$ and $k-$tuple $(i_1,\ldots,i_k)\in \{1,\ldots,m\}^k$ there exists $C_{s,k}>0$ depending only on $k,s,G,\sigma_0, \Lambda$ such that
 depending on $G,\sigma_0, \Lambda$ but independent of $\e$ such that for each $A_1,A_2\in M_\Lambda$, $\e>0$, $\bar x\in \bar G$ and $t>0$ one has
\begin{equation}\label{giovanna1}
|(\p_t^s {\bar X}^\e_{i_1}\cdots {\bar X}^\e_{i_k} \bar \Gamma_{\e, A_1})(\bar x,t) - (\p_t^s { \bar X}^\e_{i_1}\cdots {\bar  X}^\e_{i_k} \bar\Gamma_{\e, A_2})(\bar x,t) |\leq
\end{equation}
$$ \le ||A_1- A_2||C_{s,k} t^{-s-k/2} \frac{e^{-\frac{\bar d_\e(\bar x,0)^2}{C_\Lambda t}}} {|\bar B_\e(0, \sqrt{t})|}.$$
\end{lemma}

\begin{proof}
In view of \eqref{distances} and  (\ref{changegamma}) it is sufficient to  establish the result for $\bar \Gamma_{0,A_1} $ and $\bar \Gamma_{0,A_2}$, thus eliminating the dependence on $\e$. 

Although the vector fields $\bar X^0_i$, $i=1,...,m+n$ are not free, we can invoke the Rothschild and Stein lifting theorem \cite{Roth:Stein} and lift them to a new family of free vector fields in a free Carnot group. For the sake of simplicity we will continue to use the same notation $\bar X^0_i$, $i=1,...,m+n$ to denote such family of free vector fields. To recover the desired estimate  \eqref{giovanna1} for the original vector fields  one needs to argue through projections down to the original space, exactly as done in \cite{nsw} and \cite{CiMa-F}. A standard argument (see for instance \cite{BLU2}) yields that for every $A\in M_\Lambda$, if we set $\bar A$ as in \eqref{giovanna}  then there exists 
a Lie group authomorphism $F_A$ such that 
$$(\sum_{j=1}^{m+n} ( \bar A^{1/2})_{ij} \bar X^0_j )(u \circ F_A) =(\bar X_i u)\circ F_A,$$ 
and
 $$\bar \Gamma_{0,A} (\bar x, t)= |\det \bar A^{1/2}| \bar \Gamma_{0,I} (F_A(\bar x), t) $$
where $I$ denotes the identity matrix. The authomorphism $F_A $ is defined   by
$$F_A (\bar x)= \exp( T_A( \log(\bar x)), \text{ with } T_A(\bar X^0_i) = \sum_{j=1}^{m+n} (A^{1/2})_{ij}\bar X^0_j \text{ for } i=1,\ldots,m$$
and is extended to the whole Lie algebra as morphism (see the Appendix for more details).
As in \eqref{canonical},  we denote by $(v_i, w_i)$ the canonical coordinates of $F_{A_1}(\bar x)$ around 
$F_{A_2}(\bar x)$, then the Mean Value Theorem  yields
$$ (\p_t^s { \bar X}^0_{i_1}\cdots{\bar  X}^0_{i_k}  \bar \Gamma_{I})(F_{A_1}(\bar x), t)-(\p_t^s {\bar  X}^0_{i_1}\cdots{ \bar X}^0_{i_k}\bar\Gamma_{I})(F_{A_2}(\bar x), t)= $$$$
 \sum_{i=1}^{n} (v_i X_i + w_iY_i)(\p_t^s { \bar X}^0_{i_1}\cdots { \bar X}^0_{i_k}\bar\Gamma_{I})(\bar y, t).$$
By the result in the Appendix, the operators $v_i X_i + w_iY_i$ are zero order differential operators 
whose coefficients can be estimated by $||A_1 - A_2||$. The conclusion follows  by virtue of   Proposition \ref{TAestimate}.
\end{proof}

We conclude this section with the proof of the main result Proposition \ref{uniform heat kernel estimates}.

\bigskip

\noindent{\it Proof of Proposition \ref{uniform heat kernel estimates}.} 
 From the definition of fundamental solution 
we  have that   $$\Gamma_{0,A}( x,t)=\int_G \bar \Gamma _{0,A}((x, y),t) dy, \quad  \text{ and } \quad\Gamma_{\e, A} (x,t)=\int_G \bar \Gamma_{\e, A} ((x, y),t) dy, $$
for any $x\in G$  and $t>0$.
In view of the (global) dominated convergence of  the derivatives of $\bar \Gamma_{\e, A}$ to  the corresponding derivatives of $\bar \Gamma_{0,A}$ as $\e\to 0$, we deduce that  $$\int_G \bar \Gamma_{\e, A} ((x, y),t) dy \rightarrow \int_G \bar \Gamma_{0,A} (( x, y),t) dy$$ as   $\e\rightarrow  0$.
The Gaussian estimates of $ \Gamma_{\e, A}$ follow from the corresponding estimates on $\bar \Gamma_{\e, A}$ and the fact that in view of \eqref{tocheck}, $$\bar d_\e((x, y), (x_0, y_0))\geq \bar d_0((x, y), (x_0, y_0)) - C_0\geq  d_0( x, x_0) +  d_0(y, y_0) - C_0\geq$$
$$\geq  d_\e(x, x_0) +  d_\e(y, y_0) -3 C_0$$
Indeed the latter shows that there exists a constant $C>0$ depending only on $G, \sigma_0$ such that for every $x\in G$,
$$\int_G e^{- \frac{d^2_\e((x, y), (x_0, y_0))}{t}}dy \leq  C e^{-\frac{d^2_\e(x, x_0)}{t}}
 \int_G e^{- \frac{d^2_\e( y , y_0)}{t}}dy \leq  C e^{-\frac{d^2_\e(x, x_0)}{t}}.$$ 
The conclusion follows at once

\section{Gradient estimates}

In this section we prove Theorem \ref{global estimates}. The proof is carried out in two steps: First we use the maximum principle to establish interior $L^\infty$ bounds for the full gradient of the solution $\nabla_1 u$ of \eqref{ivp} with respect to the Lipschitz norm of $u$ on the parabolic boundary. Next, we construct appropriate barriers and invoke the comparison principle established in \cite{CC} to prove boundary gradient estimates. The combination of the two will yield the uniform global Lipschitz bounds.

\subsection{Interior gradient estimates} Recalling that the right invariant vector fields $X_j^r$ commute with the left invariant frame $X_i$, $i=1,\ldots ,n$ it is easy to show through a direct computation the following result.

\begin{lemma}\label{derivatedestre}
Let $u_\e\in C^3(Q)$ be a solution to \eqref{pdee} and denote
 $v_0=\partial_t u$, $v_i = X_i^{r}u$ for $i=i, \ldots, n$. Then
 for every $h=0,\ldots ,n$ one has that $v_h$ is a solution of

\begin{equation}\label{diff-eq}\p_t v_h= X_i^\e ( a_{ij} X_jv_h )= a_{ij}^\e(\nabla_\e u_\e) X_i^\e X_j^\e v_h + \p_{\xi_k} a_{ij}^\e(\nabla_\e u)X_i^\e X_j^\e u_\e X_k^\e v_h,\end{equation}
where
$$ a_{ij}^\e(\xi) = \frac{1}{\sqrt{1+|\xi|^2}}\Big(\delta_{ij}- \frac{\xi_i \xi_j}{1+|\xi|^2}\Big).$$
\end{lemma}

Note that in order to prove $L^\infty$ bounds on the horizontal gradient of solutions of \eqref{ivp} one {\it cannot} invoke Lemma \ref{derivatedestre} with differentiation the horizontal frame, because
these vector fields do not commute. The argument below bypasses this difficulty by using the algebraic relation between right and left invariant vector fields.

\begin{prop}\label{u_t}
Let $u_\e\in C^3(Q)$ be a solution to \eqref{ivp} with $\Om$ bounded. There exists $C=C(G,||\varphi||_{C^2( \Om )})>0$
such that for every compact subset $K\subset \subset \Om$ one has
$$\sup_{K \times [0,T)} |\nabla_1 u|  \leq \sup_{\partial_p Q}(|\nabla_1 u| + |\partial _ t u|),$$
where $\nabla_1$ is the  full $\s_1-$Riemannian gradient.
\end{prop}
\begin{proof} The proof is similar to the argument in \cite[Proposition 5.1]{CC} and is based on the trivial observation that left-invariant vector fields at the layer $k$ can be written as linear combination of right-invariant vector fields at level $k$ and left (or right) invariant vector fields at levels $k+1, k+2, \ldots , r$. The first step consists in  differentiating the equation in \eqref{ivp} in the direction of the highest layer $V^r$ (which commutes with all the Lie algebra) and invoke 
 the weak maximum principle applied to \eqref{diff-eq} to show
 \begin{equation}\label{stepr}
 \sup_{K\times[0.T)} \sum_{d(i)=r} |X_i u| \le \sup_{\partial_p Q}(|\nabla_1 u| + |\partial _ t u|).
 \end{equation} Next one proceeds in a similar fashion, iterating the argument by differentiating the equation in \eqref{ivp} along the right invariant vector fields in the layer $V^{r-1}$. The maximum principle and \eqref{diff-eq} yield $L^\infty$ bounds on $X_i^ru$ with $d(X_i^r)=r-1. $
 Coupling such bounds with \eqref{stepr} one obtains
  \begin{equation}\label{stepr-1}
 \sup_{K\times[0.T)} \sum_{d(i)=r-1} |X_i u| \le \sup_{\partial_p Q}(|\nabla_1 u| + |\partial _ t u|).
 \end{equation} 
 The proof now follows by induction on the step of the derivatives. \end{proof}

%

\bigskip

\subsection{Linear barrier functions}

In \cite[Section 4.2]{CC} it is shown that in a step two Carnot group coordinate planes (i.e. images under the exponential of level sets of the form $x_k=0$) solve the minimal surface equation $h_0=0$. In the same paper
it is also shown that this may  fail for step three or higher. In the construction of the barrier function  we will need the following slight refinement of this result,

\begin{lemma}\label{planes are flat}
Let $G$ be a step two Carnot group. If $f:G\to \R$ is linear (in exponential coordinates) then for every $\e\ge 0$,  the matrix with entries
$X_i^\e X_j^\e f$ is anti-symmetric, in particular every level set of $f$ satisfies $h_\e=0$.
\end{lemma}
\begin{proof}
We need to show that for $f(x)=\sum_{i=1}^n a_i x_i$,
\begin{equation}
\sum_{i,j=1}^{n} \Bigg(\delta_{ij}-\frac{X_i^{\e }f X_j^{\e }f}{1+|\nabla_{\e} f|^2}\Bigg) X_i^{\e }X_j^{\e } f=0.
\end{equation}
We recall the expression \eqref{bch}
for the vector fields $X_i$, $d(i)=1$ in terms of exponential
coordinates
$X_i=\p_{x_i}+\sum_{d(j)=1, d(h)=2} c_{ij}^h x_j \p_{x_h}.$
 The Campbell-Hausdorff formula
 implies  the anti-symmetry relation $c_{ij}^h =
-c_{ji}^h$. It is immediate to observe that, if $d(k)=2$ one
has
\begin{equation}\label{gradplanes}
X_i (x_k)= \sum_{d(j)=1}c_{ij}^kx_j, \text{ and }X_i X_j
(x_k)=c_{ji}^k, \text{ for }d(i)=d(j)=1,
\end{equation}
if either $d(i)=2$ or $d(j)=2$ it is easy to check $X_i^\e X_j^\e (x_k)=0$
for all $k=1,\ldots ,n$. Since $\Bigg(\delta_{ij}-\frac{X_i^{\e }f X_j^{\e }f}{1+|\nabla_{\e} f|^2}\Bigg)$ is symmetric in $i,j$ it follows that
\begin{equation}
h_\e= \sum_{k=1}^n a_k \sum_{i,j=1}^{n} \Bigg(\delta_{ij}-\frac{X_i^{\e }f X_j^{\e }f}{1+|\nabla_{\e} f|^2}\Bigg) X_i^{\e }X_j^{\e } (x_k)=0.
\end{equation}
\end{proof}

\subsection{Boundary gradient estimates}  
We say that a set $\Om\subset G$ is convex in the Euclidean sense if
$\exp^{-1}(\Om)\subset \lie$ is convex in the Euclidean space $\lie$.
In a group of step two this condition is translation invariant.

\begin{prop}\label{boundary estimates}
Let $G$ be a Carnot group of step two,
 $\Om\subset G$ a bounded, open, convex (in the Euclidean sense) set and $\varphi\in C^2(\bar\Om)$. For $\e>0$ denote by  $u_\e\in C^2(\Om\times(0,T))\cap C^1(\bar\Om\times (0,T))$ the non-negative unique  solution of the initial value problem \eqref{ivp}. There exists $C=C(G, ||\varphi||_{C^2(\bar\Om)})>0$ such that
 \begin{equation}\label{b-1.1}
 \sup_{\p\Om\times(0,T)}|\nabla_\e u_\e| \le  \sup_{\p\Om\times(0,T)}|\nabla_1 u_\e| \le C.
 \end{equation}
 \end{prop}

We start by recalling an immediate consequence of the proof of \cite[Theorem 3.3]{CC}.

\begin{lemma}\label{ESth32bis}
For each $\e\ge 0$,
if  $U_\e$ is a bounded  subsolution and $V_\e$ is a bounded
 supersolution of \eqref{ivp} then $U_\e(x,t)\leq V_\e(x,t)$ for all $(x,t)\in Q$.
\end{lemma}

Let $u_\e\in C^2(Q)$ be a solution of \eqref{ivp}, and express the evolution PDE in non-divergence form
\begin{equation}\label{boundary 1.1}
\p_t u_\e =h_\e   = \sum_{i,j=1}^n a_{ij}^\e (\nabla_\e u) X_i^\e X_j^\e u_\e.
\end{equation}
Set $v_\e=u_\e-\varphi$ so that $v_\e$ solves the homogenous 'boundary' value problem
 \begin{equation}\label{ivph}
 \Bigg\{
 \begin{array}{ll}
 \p_t v_\e=  a_{ij}^\e (\nabla_\e v_\e+\nabla_\e \varphi) X_i^\e X_j^\e v_\e   +b^\e &\text{ in }Q=\Om\times(0,T) \\
 v_\e=0 &\text{ on } \p_p Q,
 \end{array}
 \end{equation}
with $b^\e(x)= a_{ij}^\e(\nabla_\e v_\e (x)+\nabla_\e \varphi(x) ) X_i^\e X_j^\e\varphi(x).$
We define our (weakly) parabolic operator for which the function $v_\e$ is  a solution
\begin{equation}\label{boundary 1.2}
Q(v)= a_{ij}^\e (\nabla_\e v_\e+\nabla_\e \varphi) X_i^\e X_j^\e v_\e   +b^\e -\p_t v.
\end{equation}

In the following we construct for each point $p_0=(x_0,t_0)\in \p\Om \times (0,T)$ a {\it barrier function} for $Q, v_\e$:  i.e.,

\begin{lemma}\label{barrier-lemma}
Let  $G$ be a Carnot group of step two and  $\Om\subset G$  convex in the Euclidean sense. For  each point $p_0=(x_0,t_0)\in \p\Om \times (0,T)$ one can construct a positive function $w\in C^2(Q)$ such that
\begin{equation}\label{boundary 1.3} Q(w)\le 0 \text{ in }V\cap Q \text{ with }V\text{
a parabolic neighborhood of } p_0,\end{equation}$$ w(p_0)=0 \text{ and }w\ge v_\e\text{ in }\p_pV\cap Q.$$

\end{lemma}

\begin{proof} In the hypothesis that $\Om$ is convex in the Euclidean sense we have that
every $x_0\in \p \Om$ supports a tangent hyperplane $P$ defined by an equation of the form $\Pi(x)=\sum_{i=1}^n a_i x_i=0$ with $\Pi>0$ in $\Om$, $\Pi(x_0)=0$,
and normalized as $\sum_{d(i)=1,2}a_i^2=1$.
Following the standard argument (see for instance \cite[Chapter 10]{Lieberman}) we select the barrier at $(x_0,t_0)\in \p \Om\times(0,T)$ independent of time with
\begin{equation}\label{boundary 1.4}
w= \Phi (\Pi)
\end{equation}
with $\Phi$ solution of \begin{equation}\label{ode}\Phi''+\nu (\Phi')^2=0,
\end{equation}
 in particular
\begin{equation}\label{ode-sol}\Phi(s)= \frac{1}{\nu} \log (1+ks),
\end{equation} with $k$ and $\nu$ chosen appropriately so that conditions
\eqref{boundary 1.3} will hold. We choose a neighborhood $V=O\times (0,T)$ such that $P\cap O\cap\p\Om=\{x_0\}$. By an appropriate choice of $k$ sufficiently large we can easily obtain $w_\e(p_0)=0 \text{ and }w_\e\ge v_\e\text{ in }\p_pV\cap Q.$

To verify $Q(w_\e)\le 0$ we begin by observing that $w$ satisfies
\begin{equation}\label{boundary 1.5}
Q(w)= \Phi' a_{ij}^\e (\nabla_e w+\nabla_\e \varphi) X_i^\e X_j^\e \Pi + \frac{\Phi''}{(\Phi')^2}\F +b_\e,
\end{equation}
with $\F= a_{ij}^\e (\nabla_e w+\nabla_\e \varphi) X_i^\e w X_j^\e w.$

We will show:

\begin{equation}\label{boundary 1.6}
     a_{ij}^\e (\nabla_e w+\nabla_\e \varphi) X_i^\e X_j^\e \Pi \le 0 \tag{Claim 1}
\end{equation}

\begin{equation}\label{boundary 1.7}
  \frac{\Phi''}{(\Phi')^2}\F +b_\e \le 0 \tag{Claim 2}
\end{equation}
in a parabolic neighborhood of $p_0$.

The first claim holds with an equality as $ a_{ij}^\e$ is symmetric and
$X_i^\e X_j^\e \Pi$ is anti-symmetric in view of Lemma \ref{planes are flat}.
To establish \eqref{boundary 1.7} we first note that  Lemma   \ref{planes are flat} implies  $$\frac{\e^2}{2}\le  \max(\sum_{d(i)=1} a_i^2, \e^2 \sum_{d(k)=1} a_k^2) \le |\nabla_\e \Pi |=$$$$=\sum_{d(i)=1} \bigg( a_i + \sum_{d(k)=2, d(j)=1} c^k_{ij} a_k x_j\bigg)^2+ \e^2 \sum_{d(k)=2} a_k^2 \le C(G) (1 + \e^2), $$ for some constant $C(G)>0$.
Consequently, for $\Phi'>>1$ sufficiently large one finds
\begin{equation}\label{boundary 1.8}
\F\ge \frac{|\nabla_\e w|^2}{(1+|\nabla_\e w+\nabla_\e \varphi|^2)^{3/2}} \ge C(G) \frac{|\nabla_\e w|^2}{(1+|\Phi'|^2 +|\nabla_\e \varphi|^2)^{3/2}} \ge C(G) \e^2>0,
\end{equation}
with $C(G)>0$ a constant depending only on $G$ (not always the same along the chain of inequalities). In view of the definition of $b_\e$ and \eqref{ode} with an appropriate choice of $\nu=\nu(G,\e, \phi)>0$ and $k=k(G,\phi)>>1$ in \eqref{ode-sol}, we conclude
\begin{equation}\label{boundary 19}
 \frac{\Phi''}{(\Phi')^2}\F +b_\e \le \bigg(  \frac{\Phi''}{(\Phi')^2} +\nu\bigg) \F=0.
\end{equation}
\end{proof}

In view of Lemma \ref{ESth32bis}, a comparison with the barrier constructed above yields that
\begin{equation}\label{boundary 20}
0\le \frac{v_\e (x,t)}{dist_{\sigma_1}(x,x_0)} \le  \frac{w(x,t)}{{dist_{\sigma_1}(x,x_0)}} \le C(k,\nu),
\end{equation}
in $V\cap Q$, with $dist_{\sigma_1}(x,x_0)$ being the distance between $x$ and $x_0$ in the Riemannian metric  $\sigma_1$,  concluding the proof of the boundary gradient estimates.

\bigskip

The proof of Theorem \ref{global estimates} now follows immediately from Proposition \ref{u_t} and Proposition \ref{boundary estimates}.

\bigskip

Having established uniform global Lipschitz bounds one now notes  that equation \eqref{diff-eq} satisfies {\it horizontal} coercivity conditions uniformly in $\e>0$. Such conditions are among the main hypothesis of the Harnack inequality in \cite{CCR}. In this paper, G. Rea and two of the authors have proved that given a homogenous structure and a Poincar\'e inequality, then  a sub-elliptic analogue of Aronson and Serrin's Harnack  inequality \cite{AR} for quasilinear parabolic equations holds. As a consequence one obtains for some $\al\in (0,1)$ that the solutions $u_\e$ to \eqref{ivp} satisfy $C^{1,\alpha}$ H\"older estimates, uniform in $\e\in (0,1)$,

\begin{dfn}\label{defholder}
Let $0<\alpha < 1$, ${Q}\subset\R^{n+1}$ and  $u$ be defined on
${Q}.$ We say that $u \in C_{\e,X}^{\alpha}({Q})$ if there exists a positive constant $M$ such that for
every $(x,t), (x_{0},t_0)\in {Q}$ 
\begin{equation}\label{e301}
   |u(x,t) - u(x_{0},t_0)| \le M \tilde d_{\e }^\alpha((x,t), (x_{0},t_0)).
\end{equation}
We put  
 $$||u||_{C_{\e,X}^{\alpha}({Q})}=\sup_{(x,t)\neq(x_{0},t_0)} \frac {|u(x,t) - u(x_{0},t_0)|}{\tilde d_{\e }^\alpha((x,t), (x_{0},t_0))}+ \sup_{Q} |u|.$$
 Iterating this definition, 
 if $k\geq 1$  we say that $u \in
C_{\e,X}^{k,\alpha}({Q})$  if for all $i=1,\ldots ,m$
 $X_i \in C_{\e,X}^{k-1,\alpha}({Q})$.
 Where we have set $C^{0,\alpha}_{\e,X}({Q})=C^{\alpha}_{\e, X}({Q}).$ 
\end{dfn}

\begin{cor}\label{propCCR}{(Interior $C_X^{1,\alpha}$ estimates)}  In the hypothesis of the previous results, letting
$K$  be  a compact set  $K\subset\subset Q$, there exist constants $\al\in (0,1)$ and  $C=C(K,\al)>0$ such that for all $i=1,\ldots ,n$ one has that 
 $v= X_i^\e u $ satisfies
$$||v||_{ C_{\e,X}^{\alpha}(K)} + ||\nabla_\e v||_{L^2(K)}\leq C,$$
uniformly in $\e\in (0,1)$.
\end{cor}

\section{Regularity properties in the $C^{k, \alpha}$ spaces}

In this section we will prove uniform estimates for solution of \eqref{pdee}  in the $C^{k,\alpha}_{\e,X}$  H\"older spaces.  This is accomplished by using the uniform Gaussian bounds established in Section \ref{GaussGauss} to develop new  uniform Schauder estimates for solutions of second order sub-elliptic differential equations in non divergence form 
$$L_\e u\equiv \p_t u- \sum_{i,j=1}^n  a^\e_{ij}(x,t) X_i^\e X_j^\e u=0 ,$$
in a cylinder $ Q=\Om\times (0,T)$ that are stable as $\e\to 0$. As usual we will make use of the associated linear,
constant coefficient  {\it  frozen } operator:
$$L_{\e,(x_0,t_0)}\equiv \p_t - \sum_{i,j=1}^n a^\e_{ij}(x_0,t_0) X_i^\e X_j^\e,$$
where $(x_0,t_0)\in Q$.

\smallskip


%
As a direct consequence of the definition of fundamental solution one has the following representation formula
\begin{lemma}\label{representation} Let $w$ be a smooth solution to
$L_\e w=f $ in $Q\subset\R^{n+1}$.
For every  $\phi \in C^\infty_0(Q)$,
\begin{align}\label{repres}
   (w\phi)(x,t)
& =   \int_Q\G^\e_{(x_0,t_0)}((x,t),(y,\tau))
\big(L_{\e, (x_0,t_0)}-L_\e\big)(w \,\phi)(y,\tau) dyd\tau+\\ \nonumber 
&+\int_{Q}\G^\e_{(x_0,t_0)}((x,t),(y,\tau))
\big( f\phi  +  wL_\e\phi + 2 \sum_{i,j=1}^n a_{ij}^\e (y,\tau) X_i^\e wX_j^\e \phi\big)(y,\tau) dyd\tau,
\end{align}
where  we have denoted by $\G^\e_{(x_0,t_0)}$ the heat kernel for  of $L_{\e,(x_0,t_0)}$.
\end{lemma}

We explicitly note that for $\epsilon>0$ fixed the operator $L_{\e, (x_0,t_0)}$ is uniformly parabolic.
iIs heat kernel can be studied through standard singular integrals theory in the corresponding Riemannian balls. Hence, as noted  in \cite{Friedman},
one is allowed to  differentiate twice the kernels defined in (\ref{repres}) with respect to any right or left  invariant vector field. 

\begin{prop}\label{elle2} 
Let  $0<\alpha<1$  and  $w$ be a smooth  solution of $L_{\e}w=f\in C^{\alpha}_{\e.X}({Q})$ in the cylinder   ${Q}$.
Let $K$ be a compact sets such that  $K\subset\subset {Q}$,  set $2\delta=d_0(K, \p_p Q)$ and
denote by $K_\delta$ the $\delta-$tubular neighborhood of $K$.
 Assume that 
there exists a constant $C>0$ such that for every $\e\in (0,1)$
$$ || a^\e_{ij}||_{C^{ \alpha}_{\e,X}(K_\delta)}\leq C. $$
There exists a constant $C_1>0$ depending  on $\delta$,
$\alpha$, $C$ and the constants in Proposition \ref{uniform heat kernel estimates}
such that
$$||w||_{C^{2, \alpha}_{\e,X}(K)} \leq  C_1\left( ||f||_{C^{\alpha}_{\e,X}(K_\delta)}+ ||w||_{C^{1, \alpha}_{\e,X}(K_\delta)}\right).$$
\end{prop}

\begin{proof} The proof follows the outline of the standard case, as in  \cite[Theorem 4, Chapter 3]{Friedman}, and rests crucially on the Gaussian estimates proved in  Proposition \ref{uniform heat kernel estimates},  and  the fact that the functions $\left(L_{\e, (x_0,t_0)}-L_\e\right)(w \,\phi)$,  and $\left( f\phi  +  wL_\e\phi + 2 a_{ij}^\e X_i^\e wX_j^\e \phi\right)$  are  H\"older continuous. 

Choose a parabolic sphere\footnote{That is a sphere in the group $\tilde G=G\times \R$ in the pseudo-metric $\tilde{d}_\e$ defined in \eqref{defde}.} $B_{\e,\delta} \subset\subset K$ where  $\delta>0$ will be fixed later and 
a cut-off function $\phi\in C^\infty_0(\R^{n+1})$ identically 1 on $B_{\e, \delta/2}$ and compactly supported in $B_{\e,\delta}$. This clearly implies that for some constant $C>0$ depending only on $G$ and $\s_0$, 
$\left|\nabla_\e \phi\right|\leq C\delta^{-1}, \quad |L^\e \phi| \leq C\delta^{-2},$ in $Q$.
Next, invoking \eqref{representation} one has that for every multi-index $I=(i_1,i_2)\in \{1,\ldots ,m\}^2$  and   for every $(x_0,t_0)\in B_{\e,\delta}$
 \begin{align}\label{represDeriv}
X_{i_1}^\e X_{i_2}^\e  (w\phi)(x_0,t_0)  = &  \int_Q X_{i_1}^\e X_{i_2}^\e \G^\e_{(x_0,t_0)}(\cdot,(y,\tau))|_{(x_0,t_0)}
\left(L_{\e,(x_0,t_0)}-L_\e\right)(w \,\phi)(y,\tau) dyd\tau+\\ \nonumber 
    & + \int_Q X_{i_1}^\e X_{i_2}^\e \G^\e_{(x_0,t_0)}(\cdot,(y,\tau))
\left( f\phi  +  wL_\e \phi + 2\sum_{i,j=1}^n  a_{ij}^\e X_i^\e wX_j^\e \phi\right)(y,\tau) dyd\tau.
\end{align} 
The uniform  H\"older continuity of $a_{ij}^\e$, with Proposition \ref{uniform heat kernel estimates} and Lemma \ref{convergencelem} yield
\begin{align*}|X_{i_1}^\e X_{i_2}^\e \G^\e_{(x,t)}((x,t),(y,\tau))&- X_{i_1}^\e X_{i_2}^\e \G^\e_{(x_0,t_0)}((x_0,t_0),(y,\tau))|\\
&\le C\,\tilde d^\alpha _\e((x,t),(x_0,t_0)) 
\frac{(\tau-t_0)^{-1}e^{-\frac{d_\e(x_0,y)^2}{C_\Lambda (\tau-t_0)}}} {|B_\e(0, \sqrt{\tau-t_0})|},
\end{align*}
with $C>0$ independent of $\e$.
In view of the latter, using basic singular integral properties (see \cite{fol:1975}) and proceeding as in \cite[Theorem 2, Chapter 4]{Friedman}, 
we obtain
\begin{align}\label{estimate1}
\Big|\Big|
\int X_{i_1}^\e X_{i_2}^\e\G^\e_{(x_0,t_0)} &(\cdot,(y,\tau)) (L_\e-L_{\e,(x_0, t_0)})
(w \,\phi )(y,\tau)dyd\tau \Big|\Big|_{C^\alpha_{\e,X}(B_{\e,\delta})} \\ \nonumber
&\le C_1 \Big|\Big|(L_\e-L_{\e,(x_0, t_0)})
(w \,\phi )\Big|\Big|_{C^\alpha_{\e,X}(B_{\e, \delta})} \\ \nonumber
&= C_1\sum_{i,j}\left|\left|(a_{ij}^\e(x_0,t_0)-a_{ij}^\e(\cdot)\big)X_j^\e X_j^\e(w \,\phi)
   \right|\right|_{C^\alpha_{\e,X}(B_{\e, \delta})}\\ \nonumber
&\le \tilde C_1 \delta^\alpha  ||a_{ij}^\e||_{C^{\alpha}_{\e,X}(B_{\e,\delta})}||w\phi||_{C^{2, \alpha}_{\e,X}(B_{\e,\delta})},
\end{align}
where $C_1,$ and $\tilde C_1$ are  stable  as $\e\to 0$.
Similarly, if $\phi$ is fixed, the H\"older norm of the second term in the representation formula (\ref{represDeriv}) is bounded by
\begin{align}\label{estimate2}
\Big|\Big|\int X_{i_1}^\e X_{i_2}^\e\G^\e_{(x_{0},t_0)}((x_0,t_0),(y,\tau))
&\big(f\phi(y,\tau) + wL\phi(y,\tau) + 2 a_{ij}^\e X_i^\e wX_j^\e \phi \big)dyd\tau\Big|\Big|_{C^\alpha_{\e,X}(B_{\e, \delta})}\\  \nonumber
& \leq C_2 \left(||f||_{C^\alpha_{\e,X}(K_\delta)} +\frac{C}{\delta^2}|| w||_{C^{1, \alpha}_{\e,X}(K_\delta)}\right).
\end{align}
From  (\ref{represDeriv}), (\ref{estimate1}) and (\ref{estimate2})  we deduce that 
$$||w\phi||_{C^{2, \alpha}_{\e,X}(B_\delta)}\leq  \tilde C_2\, \delta^\alpha ||w\phi||_{C^{2, \alpha}_{\e,X}(B_\delta)} + C_2\left(||f||_{C^\alpha_{\e,X}(K_\delta)} +\frac{C}{\delta^2}|| w||_{C^{1,\alpha}_{\e,X}(K_\delta)}\right).$$
Choosing $\delta$ sufficiently small we prove  the assertion  on the fixed sphere $B_{\e,\delta}$
The conclusion follows from  a standard  covering argument.

\end{proof}

%
%


A direct calculation shows the following
\begin{lemma}\label{kderiv} Let $k\in N$ and consider a  $k-$tuple $(i_1,\ldots,i_k)\in \{1,\ldots,m\}^k$.
There exists a differential operator  $B$ of order $k+1$,   depending on horizontal derivatives of $a_{ij}^\e$ of order at most $k$, such that  
$$ X_{i_1}^\e\cdots X_{i_k}^\e  \big(L_{\e,(x_0,t_0)}-L_\e\big)= \sum_{i,j=1}^n \Big(
a_{ij}^\e - a_{ij}^\e(x_0, t_0) \Big)  X_{i_1}^\e\cdots X_{i_k}^\e X^\e_iX^\e_j  + B. $$
\end{lemma}

\begin{prop}\label{ellek}
Let  $w$ be a smooth  solution of $L^{\e}w=f$ on ${Q}$.
Let $K$ be a compact sets such that  $K\subset\subset {Q}$,  set $2\delta=d_0(K, \p_p Q)$ and
denote by $K_\delta$ the $\delta-$tubular neighborhood of $K$.Assume that 
there exists a constant $C>0$ such that 
$$ || a_{ij}^\e||_{C^{k,\alpha}_{\e,X}(K_\delta)} \leq C,$$ for any $\e\in (0,1)$.
There exists a constant $C_1>0$ depending  on
$\alpha$, $C$, $\delta$, and the constants in Proposition \ref{uniform heat kernel estimates},
but independent of $\e$,  such that
$$||w||_{C^{k+2, \alpha}_{\e,X}(K)} \leq C_1 \left( ||f||_{C^{k,\alpha}_{\e,X}(K_\delta)}+ ||w||_{C^{k+1, \alpha}_{\e,X}(K_\delta)}\right). $$
\end{prop}

\begin{proof}
The proof is similar to the  previous one for $k=1$. We note that  differentiating 
  the  representation formula (\ref {repres}) along any
  $k-$tuple $(i_1,\ldots ,i_k)\in \{1,\ldots,m\}^{k}$,
 and using Lemma \ref{kderiv}, yields
\begin{align*}
 X_{i_1}^\e\cdots X_{i_{k}}^\e    (w\phi)&(x,t)
 =   \int_Q\G^\e_{(x_{0},t_0)}((x,t),(y,\tau))
\Big(
a_{ij}^\e - a_{ij}^\e(x_0, t_0) \Big)X_{i_1}^\e\cdots X_{i_{k}}^\e    X^\e_iX^\e_j (w \,\phi)(y,\tau) dyd\tau\\
&+  \int_Q\G^\e_{(x_{0},t_0)}((x,t),(y,\tau))B (w \,\phi)(y,\tau) dyd\tau \\
&+\int_Q\G^\e_{(x_{0},t_0)}((x,t),(y,\tau))
 X_{i_1}^\e\cdots X_{i_{k}}^\e  \Big(f\phi(y,\tau) + wL_\e\phi(y,\tau) + 2 \sum_{i,j=1}^n a_{ij}^\e X_i^\e wX_j^\e\phi \Big)dyd\tau,
\end{align*}
where $\phi$ is as in the proof of Proposition \ref{elle2} and $B$ is a differential operator of order $k+1$.
The conclusion follows by further differentiating the previous representation formula along two horizontal directions $X_{j_1}^\e X_{j_2}^\e$ and arguing  as in the proof of  Proposition \ref{elle2}. 
\end{proof}

\begin{proof}
[Proof of Theorem \ref{global in time  existence results}]
We will prove by induction that for every $k\in N$ and for every compact set $K\subset\subset {Q}$ there exists a positive constant $C$ such that
\begin{equation}\label{stimak}
||u_\e||_{C^{k,\alpha}_{\e,X}(K)}\leq C,
\end{equation}
for every $\e>0$.
The assertion is true if $k=2$,  by Corollary \ref{propCCR}  and Proposition  \ref{elle2}. If (\ref{stimak}) is true for $k+1$, then the coefficients  in \eqref{pdee} satisfy $a_{ij}^\e\in C^{k, \alpha}_{\e,X}$ uniformly as $\e\in (0,1)$ and 
 (\ref{stimak}) thus holds at the level $k+2$ by virtue of Proposition  \ref{ellek}.
\end {proof}

\section{Appendix }
Let us assume that $G$ be a free nilpotent Lie group, and let $\lie=V^1\oplus...\oplus V^r$ its associated Lie algebra. 
Denote by  $X_1, \ldots, X_m$ a basis of the first layer of the Lie algebra. 
and by $X_{m+1}, \ldots, X_n$  a list of vectors which complete a basis of the tangent space. 
If $u(x)$ is an homogeneous polynomial of order $p$ and  $X$ is a differential 
operator of degree $q$ we will call $u(x) X$ differential operator of degree $q-p$. 
If $a_{ij}=(A)_{ij}$ is a real valued, constant coefficient, $m\times m$ positive definite matrix , one can define 
a Lie algebra automorphism  
$$T_A: V^1 \rightarrow V^1, 
\text{ with } T_A(X_i) = \sum_{j=1}^m a_{ij} X_j$$
for every $i=1, \ldots, m$ and extend it   to the whole Lie algebra as a morphism 
$$ T_A[X_i, X_j] = [T_A(X_i),T_A(X_j)] = \sum_{h, k=1}^m a_{ih}a_{jk}[X_h, X_k].$$
Note that 
$T_A$ can be represented as a block matrix  $\bar A = A_1 \oplus A_2 \cdots \oplus A_r$, with $A_1=A$ and 
where the block $A_j$ act on vectors of degree $j$, and its coefficients are polynomial of order $j$ of the coefficients of $A_1$. 
In particular $T_A(X)$ and $X$  have the same degree.
Via the exponential map $T_A$ induces a group automorphism on the whole group
$$F_A: G \rightarrow G, \text{ defined by } F_A(x) = \exp(T_A(\log(x)).$$
In terms of  exponential coordinates  ({\it around} the origin)
\begin{equation}\label{expcoord}x= \exp( \sum_{i=1} ^n v_i X_i)(0)\end{equation}
one has
$$F_A(x)= \exp( \sum_{i=1} ^n v_i T_A(X_i ))(0) = \exp(\sum_{s=1} ^r \sum_{d(i)=d(j) =s} v_i  (A_s)_{ij} X_j)(0),$$
where $d(i)$ denotes the homogenous degree   defined as in \eqref{degree}.
Since $ v_iX_i$ is a zero order  homogeneous operator then also $ v_i T_A(X_i )$ is a differential operator of order zero.

\begin{prop}\label{TAestimate}
Let  $(w_i)_{i=1, \ldots, n}$ denote the canonical coordinates of $F_{A}(x)$ around $F_{B}(x)$, and $(v_j)_{j=1, \ldots, n}$ the coordinates of $x$ around $0$, as in \eqref{expcoord}.  There exists $M\in\N$ depending only on the group structure; constants $c_1,\ldots,c_M$ depending only on the group structure;  zero-order differential operators $Y_1,\ldots,Y_M$ depending only on the group structure, on the coefficients $(v_j)_{j=1, \ldots, n}$ and their derivatives along vector fields up to order $r-1$ (but independent on $A$ and $B$), and a constant $C=C(||A||, ||B||, G)>0$  such that \begin{equation}\label{forma}
\sum_{i=1}^n w_iX_i = \sum _{l=1}^M c_l Y_l
\quad \text { and  } |c_l|\le C||A-B|| \text{ for } l=1,\ldots,M.\end{equation}
\end{prop}

\begin{proof}
By definition
$$ \sum_{j=1}^n w_j  X_j =  \sum_{i=1} ^n v_i T_A(X_i )  * (- \sum_{i=1} ^n v_i T_B(X_i )) $$
where $*$ is the Baker-Campbell-Hausdorff operation (see for instance \cite{Roth:Stein} and references therein). This formula allows to represent the left-hand side as a finite\footnote{since the group is nilpotent} sum of terms involving commutators (up to order $r$) of the two terms on the right-hand side. 
Let us consider separately the terms in this representation starting with those involving no commutators, i.e. 
$$\sum_{i=1} ^n v_i T_A(X_i )  - \sum_{i=1} ^n v_i T_B(X_i )= \sum_{s=1} ^r \sum_{d(i)=d(j) =s} v_i \Big( (A_s)_{ij} -  (B_s)_{ij}\Big)  X_j .$$
Clearly  this expression is in the same form as \eqref{forma}. Next we consider the second term in the Baker-Campbell-Hausdorff sum, i.e. that involving commutators of the form
$$ [ \sum_{i=1} ^n v_i T_A(X_i )  , \sum_{h=1} ^n v_h T_B(X_h )] = $$
$$ \sum_{p,s=1} ^r \sum_{d(i) =d(j)=d(h)=d(k)=s}(A_s)_{ij} (B_p)_{hk} \Big[v_i   X_j, v_h  X_k\Big] = $$
$$ \sum_{p,s=1} ^r \sum_{d(i) =d(j)=d(h)=d(k)=s} \frac{1}{2}\Big( (A_s)_{ij} (B_p)_{hk}  -  (A_p)_{hk} (B_s)_{ij}\Big)\Big[v_i   X_j, v_h  X_k\Big] .$$
It is evident that the coefficients of the commutators above are Lipschitz functions of $A$ and $B$, vanishing for $A=B$, and consequently the whole sum satisfies \eqref{forma}. 
To conclude the proof we note that all further terms of the Baker-Campbell-Hausdorff sum involve operators as in \eqref{forma} and zero order operators  of the form $\sum_{i=1} ^n v_i T_A(X_i ) $ or $- \sum_{i=1} ^n v_i T_B(X_i )$. Such commutators give rise to  differential operators of  zero order  whose  coefficients are  polynomials in $A$ and $B$ and vanish when $A=B$. \end{proof}

\end{document}